\documentclass[intlimits]{amsart}
\pdfoutput=1 

\usepackage[all,hyperref,numberbysection]{bi-discrete} 
\usepackage[backend=biber,maxbibnames=5,maxalphanames=5,style=alphabetic,bibencoding=utf8,giveninits,url=false,isbn=false]{biblatex}
\AtBeginBibliography{\small}

\usepackage{tikz}

\usepackage{esint} 
\newcommand{\dxyn}{\frac{dx\,dy}{\abs{x-y}^n}}
\DeclareMathOperator{\tail}{Tail}

\begin{document} 
	
\title[The De Giorgi method for local and nonlocal systems]{The De Giorgi method for \\local and nonlocal systems}%

\author{Linus Behn}
\address{%
	Fakult\"{a}t  f\"{u}r Mathematik,
	University of Bielefeld,
	Universit\"{a}tsstrasse~25, 33615 Bielefeld, Germany}
\email{linus.behn@math.uni-bielefeld.de}

\author{Lars Diening}
\email{lars.diening@uni-bielefeld.de}

\author{Simon Nowak}
\email{simon.nowak@uni-bielefeld.de}
\author{Toni Scharle}
\email{tscharle@rosen-group.com}


\begin{abstract}
	We extend the De Giorgi iteration technique to the vectorial setting. For this we replace the usual scalar truncation operator by a vectorial shortening operator. As an application, we prove local boundedness for local and nonlocal nonlinear systems. 
  Furthermore, we show convex hull properties, which are a generalization of the maximum principle to the case of systems.
\end{abstract}



\thanks{This work was partially funded by the Deutsche Forschungsgemeinschaft (DFG, German Research Foundation) - SFB 1283/2 2021 - 317210226 (project A7) and IRTG 2235 (Project 282638148)}

\maketitle


\section{Introduction}
\label{sec:introduction}

In his 1957 paper \cite{DeGiorgi57}, De Giorgi developed a new method to prove Hölder regularity of weak solutions to uniformly elliptic equations, today known as the \emph{De~Giorgi iteration technique}. In this paper, we present a way to extend this method to the vectorial setting, i.e., to systems of equations. This way, we are able to show local boundedness for local and nonlocal systems.

The key difficulty lies in the fact that in De Giorgi's technique estimates for the level functions $u_\lambda$ play a central role, where
\begin{align*}
	u_\lambda (x)\coloneqq \max\set{u(x)-\lambda ,0}.
\end{align*} 
Note that the expression $\max\set{u(x)-\lambda ,0}$ does not make sense if $u(x)\in \RRN$ is a vector. Thus, in order to extend the technique to the vectorial setting, we replace it by by $S_\lambda u$, where $S_\lambda$ is the \emph{shortening operator}.  This operator which we introduce below already appeared in~\cite{Scharle20,DieningScharleSueli21}. Fine properties of~$S_\lambda$ including some novel estimates are discussed in detail in Section~\ref{sec:truncation-operators}.

The shortening operator $S_\lambda\,:\, \RRN \to \RRN$ at the level of~$\lambda > 0$ is defined via
\begin{align*}
	S_\lambda a
	&\coloneqq \max \lbrace \abs{a}-\lambda, 0 \rbrace \frac{a}{\abs{a}}
	=
	\begin{cases}
		0 &\quad \text{for $\abs{a} \leq \lambda$},
		\\
		(\abs{a}-\lambda) \frac{a}{\abs{a}} &\quad \text{for $\abs{a} > \lambda$}.
	\end{cases}
\end{align*}
For a convex closed set $K\subset \RRN$, let $\Pi_K$ be the closest point projection onto~$K$. In the special case that $K=\overline{B_\lambda(0)}$ for some $\lambda>0$, we obtain the \emph{truncation operator} $T_\lambda$ at the level of $\lambda$ with $T_\lambda(a) = \min \set{\abs{a},\lambda} \frac{a}{\abs{a}}$.
The operators $\Pi_K, T_\lambda$ and $S_\lambda$ are depicted in Figure \ref{fig:PiTandS}. 
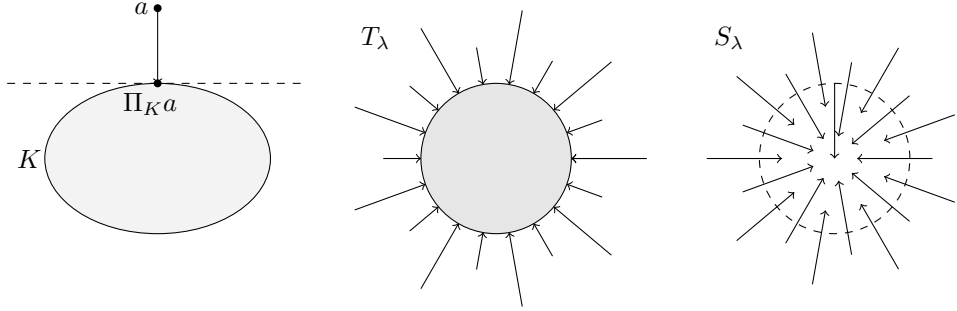
\begin{figure}[!ht]
	\centering
	
	\begin{tikzpicture}
		\begin{scope}[shift={(-4.5,0)}]
			\draw (0,0) ellipse (1.5 and 1);
			\fill[opacity=0.05] (0,0) ellipse (1.5 and 1);
			\coordinate (a) at (0,2);
			\coordinate (b) at (0.8,.5);
			\coordinate (Pia) at (0,1);
			\fill (a) circle (1.5pt);
			\fill (Pia) circle (1.5pt);
			\node[left] at (0,2) {$a$};
			\node[anchor=north] at (-0.1,1) {$\Pi_K a$};
			\node[anchor=east] at (-1.42,0) {$K$};
			\draw[->] (a) -- (Pia);
			\draw[dashed] (-2,1) -- (2,1);
		\end{scope}
		
		\draw (0,0) circle (1);
		\fill[opacity=0.1] (0,0) circle (1);
		
		\draw[->] (2,0) -- (1,0);
		\foreach \x in {0,40,80,120,160,200,240,280,320} {
			\begin{scope}[rotate=\x]
				\draw[->] (2.0,0) -- (1,0);
		\end{scope}}
		\foreach \x in {20,60,100,140,180,220,260,300,340} {
			\begin{scope}[rotate=\x]
				\draw[->] (1.5,0) -- (1,0);
		\end{scope}}
		
		\node at (-1.6,1.6) {$T_\lambda$};
		
		\begin{scope}[shift={(4.5,0)}]
			\draw[dashed] (0,0) circle (1);
			
			\draw[->] (0,1) -- (0,0);
			\foreach \x in {0,40,80,120,160,200,240,280,320} {
				\begin{scope}[rotate=\x]
					\draw[->] (1.3,0) -- (0.3,0);
			\end{scope}}
			\foreach \x in {20,60,100,140,180,220,260,300,340} {
				\begin{scope}[rotate=\x]
					\draw[->] (1.7,0) -- (0.7,0);
			\end{scope}}
			\node at (-1.4,1.6) {$S_\lambda$};
		\end{scope}
		
	\end{tikzpicture}
	\caption{Projection~$\Pi_K$, truncation operator $T_\lambda$ and shortening operator $S_\lambda$}
	\label{fig:PiTandS}
\end{figure}

\subsection{Setup and main results}

Before stating our main results let us briefly introduce the local and nonlocal vectorial problems we are studying. 

\subsubsection{Local vectorial problems}
\label{ssec:local-vect-probl}

We begin with the local case. Let $\Omega \subset \RRn$ be open and bounded and let $\phi$ be an N-function. The necessary background on N-functions and Simonenko indices is collected in Section \ref{sec:orlicz}. An important example is given by $\phi(t)=\frac 1p t^p$ for $1< p <\infty$. We aim to show boundedness of minimizers of the functional
\begin{align}\label{eq:local_energy}
	\mathcal{J}_\phi (v) \coloneqq \int_\Omega \phi (\abs{\nabla v}) \,dx ,
\end{align}
where we minimize over all $v\in W^{1,\phi}_{u_D}(\Omega,\RRN)\coloneqq W_0^{1,p}(\Omega,\RRN)+u_D$, where $u_D\in W^{1,\phi}(\Omega,\RRN) $ are some prescribed boundary data. If $\phi(t)=\frac 1p t^p$, then $\mathcal{J}_\phi$ becomes the usual $W^{1,p}(\Omega,\RRN)$-energy.

The Euler-Lagrange system corresponding to \eqref{eq:local_energy} is given by
\begin{align}\label{eq:local_eq}
	\int_\Omega \phi'(\abs{\nabla u}) \frac{\nabla u}{\abs{\nabla u}} :\nabla \psi \,dx=0\qquad\text{for all $\psi \in W^{1,\phi}_0(\Omega,\RRN)$}.
\end{align}
In the case $\phi(t)=\frac 1p t^p$, this becomes the standard $p$-Laplace system. Our main theorem in the local case asserts that solutions to \eqref{eq:local_eq} are locally bounded. Although this property is well known, our method leads to a simpler proof of local boundedness. For a more detailed discussion we refer to Section~\ref{sec:previous_results} below

\begin{theorem}[$L^\infty$-estimates for local systems]\label{thm:orlicz_L_infty}
	Let $\Omega \subset \RRn$ be open and bounded. Let $B$ be a ball with radius $r$ such that $2B\subset \Omega$. Let $\phi$ be an N-function with Simonenko indices $p$ and $q$, where $1\leq p\leq q<\infty$. Let $u\in W^{1,\phi}(\Omega, \RRN)$ be a weak solution of \eqref{eq:local_eq}. Then $u\in L^\infty(B,\RRN)$ and we have the estimate	
	\begin{equation}
		\sup_B \phi(r^{-1}\abs{u}) \leq c \fint_{2B} \phi(r^{-1}\abs{u}) \,dx ,
	\end{equation}
	where $c$ depends only on $q$ and $n$.
\end{theorem}

\subsubsection{Nonlocal vectorial problems}
\label{ssec:nonl-vect-probl}

The main focus of this paper lies in the study of vectorial nonlocal problems analogous to \eqref{eq:local_energy}. For this let $\phi$ again be an N-function and $\Omega\subset \RRn$ open and bounded. We define for $s \in (0,1)$
\begin{align}\label{eq:nonlocal_energy}
	\mathcal{J}_\phi^s (v) \coloneqq \int_{\RRn}\int_{\RRn} \phi\left(\frac{\abs{v(x)-v(y)}}{\abs{x-y}^s}\right) \dxyn .
\end{align}
In the case $\phi(t)=\frac 1p t^p$, this becomes the $W^{s,p}(\mathbb{R}^n,\mathbb{R}^N)$ energy.
This time we assume that $u$ minimizes $\mathcal{J}_\phi^s$ among all $v\in W^{s,\phi}(\RRn,\RRN)$ with $v=g$ on $\Omega^c$, where $g\in W^{s,\phi}(\RRn,\RRN)$ is some given complement data.

Such minimizers satisfy the Euler-Lagrange equation
\begin{align}\label{eq:nonlocal_equation}
	\int_{\RRn}\int_{\RRn} \phi' \left(\frac{\abs{u(x)-u(y)}}{\abs{x-y}^s}\right) \frac{u(x)-u(y)}{\abs{u(x)-u(y)}} \frac{\psi(x)-\psi(y)}{\abs{x-y}^s} \dxyn = 0,
\end{align}
for all $\psi \in W^{s,\phi}_c (\Omega,\RRN)$. In the model case $\phi(t)=\frac 1p t^p$, \eqref{eq:nonlocal_equation} becomes the fractional $p$-Laplace system.  In order to control the decay of $u$ at infinity we introduce for any ball $B\subset \RRn$ with radius $r>0$ and center $x_B$ the \emph{nonlocal tail} of $u$ as
\begin{align*}
	\tail(u,B)\coloneqq r^s\left(\phi'\right)^{-1}\Bigg(r^s \int_{B^c}\phi'\left(\frac{\abs{u(y)}}{\abs{y-x_B}^s}\right) \frac{dy}{\abs{y-x_B}^{n+s}}\Bigg).
\end{align*}
This agrees with the definition in~\cite[(3.4)]{ChakerKimWeidner2022}.

Our main result is again concerned with local boundedness of solutions to \eqref{eq:nonlocal_equation}. In contrast to Theorem \ref{thm:orlicz_L_infty} in the local setting, this result is new.

\begin{theorem}[$L^\infty$-estimates for nonlocal systems]\label{thm:local_boundedness_nonlocal}
	Let $\Omega\subset \RRn$ be an open and bounded set. Let $B$ be a ball such that $2B\subset \Omega$. Let $\phi$ be an N-function with Simonenko indices $p$ and $q$, where $1\leq p\leq q <\infty$. If $u\in W^{s,\phi}(\RRn,\RRN)$ is a weak solution to \eqref{eq:nonlocal_equation}, then $u\in L^\infty (B,\RRN)$ and we have the estimate
	\begin{align}
		\sup_B \phi (r^{-s}\abs{u}) \leq c \fint_{2B}\phi\left(r^{-s}\abs{u}\right)\,dx + c\,\phi \big(r^{-s}\tail (u,B)\big),
	\end{align}
	where $c>0$ depends only on $q$, $s$ and $n$.
\end{theorem}

\begin{remark}
  \label{rem:moritz-tail-spaces}
	Theorem \ref{thm:local_boundedness_nonlocal} remains valid if the assumption $u\in W^{s,\phi}(\RRn,\RRN)$ is replaced by the weaker assumption $u\in W^{s,\phi}(2B ,\RRN)$ and $\tail (u,B) < \infty$.

  To allow for local minimizers relaxing the global condition $u\in W^{s,\phi}(\RRn,\RRN)$ it is better to consider the renormalized energy
  \begin{align*}
    \widetilde{\mathcal{J}}_\phi^s (v) \coloneqq \iint_{(\Omega^c \times \Omega^c)^c} \phi\left(\frac{\abs{v(x)-v(y)}}{\abs{x-y}^s}\right) \dxyn .
  \end{align*}
  This energy is minimized over all~$v$ such that $\widetilde{\mathcal{J}}_\phi^s(v)<\infty$ with $v=g$ on~$\Omega^c$. This approach has been introduced in~\cite{ServadeiValdinoci12,FelsingerKassmannVoigt15}. Our results remain valid in this situation.
\end{remark}

\subsection{Previous results}\label{sec:previous_results}

\subsubsection{Local case}

Since De Giorgi introduced his iteration method in 1957 (see \cite{DeGiorgi57}) it has proved flexible enough to be adapted by many authors to a great number of different settings. Here we focus on the generalizations most closely related to our work. In \cite[Chapters 4.7 and 5.3]{LadyUralt68} the technique was applied to nonlinear equations of $p$-Laplace type. Later, this was extended by Liebermann \cite{Liebermann91} to equations of more general Orlicz-type growth.
Moreover, in \cite{AquileraCaffarelli86,DieningScharleSueli21} the technique was even applied to finite element approximations of elliptic equations. All these are applications of the scalar De Giorgi technique.

The scalar technique can also be used on the level of the gradient to show $C^{1,\alpha}$ regularity. In this case the De Giorgi technique is applied to the scalar quantity $\abs{\nabla u}$ and can thus be used even for vectorial solutions. For systems of $p$-Laplace type this was shown by Uhlenbeck in \cite{Uhlenbeck77} for $2\leq p<\infty$  and by Tolksdorf \cite{Tolksdorf84} and Acerbi and Fusco \cite{AcerbiFusco89} for $1<p\leq 2$. 
For functionals with general Orlicz-type growth this result was obtained in \cite{DieningStroffVerde09}. The method was applied to the corresponding parabolic system~\cite{DieningScharleSchwarzacher19,OkScillaStroffolini24}.

Recently, De Giorgi-type iteration methods were also applied to obtain Schauder-type estimates for elliptic equations with general $p,q$-growth, see \cite{DFM1,DFM2}.

\subsubsection{Nonlocal case}

Let us first mention some works where De Giorgi type iteration methods have been applied to obtain local boundedness results for scalar nonlocal equations.  The De Giorgi type iteration was applied in~\cite{CaffChanVass11} to linear nonlocal equations.
For nonlinear nonlocal equations of fractional $p$-Laplace type, local boundedness was first obtained in \cite{DiCastroKuusiPalatucci2016}. The approach was axiomatized via fractional De Giorgi classes in \cite{Cozzi17} leading to corresponding results for slightly more general nonlinear nonlocal problems. The method was extended to nonlocal problems of general Orlicz type growth independently in the papers \cite{ChakerKimWeidner2022} and \cite{ByunKimOk2023}. It was also applied in~\cite{KassmannWeidner22arxivII,KassmannWeidner23arxiv} to parabolic problems.

The literature related to regularity theory of nonlocal vectorial problems is way less developed. In, for instance, \cite{DaLioRivi11,Schikorra15,MillotSire15,MillotPegonSchikorra21} the regularity of fractional harmonic mappings into spheres was investigated. In \cite{CaffDavila19}, $C^{0,\alpha}$ regularity was shown for bounded solutions of systems driven by linear translation invariant nonlocal operators with nonlinear right hand side depending on the solution. Nonlocal systems also arise in \emph{Peridynamics}, a nonlocal model of continuum mechanics. These systems incorporate nonlocal versions of the symmetric gradient appearing frequently in local models of continuum mechanics. Some regularity results for linear systems in this field were shown, e.g., in \cite{KassmannMengeshaScott19,MengeshaSchikSeesaneaYeepo24}.

Let us conclude this section by remarking that to the best of our knowledge no higher regularity results beyond local boundedness are known in the case of nonlinear nonlocal systems. This is in contrast to the scalar nonlinear nonlocal case where various higher regularity results are known, see for instance \cite{CaffChanVass11,DiCastroKuusiPalatucci2016,BrascoLindgSchikorra18,DeFilippisPalatucci19,ChakerKimWeidner2022,ByunOk22,BonderSalortVivas22,ByunKimOk2023} for $C^{0,\alpha}$-regularity, \cite{DicastroKuusiPala14,ChakerKimWeidner23} for Harnack inequalities, \cite{KuusiMingioneSire15,KimLeeLee23,DieningKimLeeNowak24} for potential estimates and \cite{BrascoLindgren17,Nowak23,DieningNowak23,ByunKim23} for higher Sobolev regularity. For these reasons, an interesting question for future investigation is whether higher regularity results, especially Hölder regularity, hold for nonlinear nonlocal systems.

\subsection{Convex hull property}

Besides local boundedness we show a \emph{convex hull property} both in the local and the nonlocal case. It states that if $u$ is a minimizer of \eqref{eq:local_energy} or \eqref{eq:nonlocal_energy}, then the values that $u$ attains in the interior of $\Omega$ are always contained in the (closed) convex hull of the values attained on $\partial \Omega$.  This is a generalization of the maximum principle to the vectorial setting. Indeed, if $u:\RRn \rightarrow \setR$, i.e., in the scalar setting, then the convex hull of the values attained at the boundary is given by the interval
\begin{align*}
	\big[\inf _{x\in \partial \Omega} u(x),\sup_{x\in \partial \Omega} u(x)\big]\subset \setR ,
\end{align*} 
and the property that $u$ attains only values inside this interval becomes the usual maximum principle. 

The convex hull properties are stated in Lemma \ref{lem:convexhull_local} and Lemma \ref{lem:convexhull_nonlocal}. Their proofs are short and allow for very general $\phi$, the only requirement being that uniqueness of minimizers of $\mathcal{J}_\phi$ resp. $\mathcal{J}_\phi^s$ is ensured.

In the local case, convex hull properties have been proven in \cite{OttavioLeonettiMusciano96,BildhFuchs02}, where they are used to obtain regularity results. In \cite{DieningKreuzerSchwarzacher2013} the convex hull property is extended to finite element minimizers of $\mathcal{J}_\phi$.

The rest of this paper is structured as follows. In Section \ref{sec:truncation-operators} we discuss the projection, truncation and shortening operator. The necessary background on N-functions and Orlicz spaces is collected in Section \ref{sec:orlicz}. In Section \ref{sec:convex-hull-property-1} we show convex hull properties in the local and nonlocal case. Section \ref{sec:de-giorgi-method} contains the De Giorgi technique and the proof of our main results. For that we need improved \Poincare~estimates which we show in the Appendix \ref{sec:appendix}.

\section{Projections, truncations and shortening operators}
\label{sec:truncation-operators}

In this section we introduce three useful operators: a projection, a truncation operator and a shortening operator. The projection is used later in the convex hull property. The truncation is a special case of the projection. The shortening operator will be used in the De~Giorgi technique. The operators are depicted in Figure~\ref{fig:PiTandS}.

We begin with a bit of standard notation.  When we write $a\lesssim b$ we mean that there exists a constant $c>0$, independent of all important appearing quantities, such that $a\leq cb$.  For a scalar function~$f$ we define $f_+(x) \coloneqq \max \set{f(x),0}$.  If $U\subset \RRn$ is a measurable set, we denote by $\abs{U}$ the Lebesgue measure $\mathcal{L}^n(U)$. If $\abs{U}\leq \infty$ and $f\in L^1(U)$ we define
\begin{align*}
	\mean{f}_{U}\coloneqq \fint_U f(x) \,dx \coloneqq \frac{1}{\abs{U}}\int_U f(x)\,dx.
\end{align*}
Let $B\subset \RRn$ be a ball. Usually we denote the radius of $B$ with $r$. For $\lambda >0$, we denote by $\lambda B$ the ball with the same center and radius $\lambda r$.

We use the convention of \emph{standing gradients}, i.e., if $u:\RRn \rightarrow \setR$, then $\nabla u (x)$ is a column vector for every $x$. Accordingly, if $u:\RRn \rightarrow \RRN$ then $(\nabla u(x) )_{ij}=\partial_i u_j(x)$ for all $1\leq j \leq n$, $1\leq j \leq N$ and $x\in \RRn$. For matrices $A,B\in \RRnn$, we denote $A:B\coloneqq \sum_{1\leq ij\leq n} A_{ij}B_{ij}$ and for vectors $v,w\in \RRn$ we define the outer product $u\otimes v\coloneqq (u_iv_j)_{1\leq i,j\leq n}$.

\subsection{Projection to convex, closed sets}
\label{sec:proj-conv-clos}

For a convex, closed set $K \subset \RRN$ we define the projection operator $\Pi_K\,:\, \RRN \to K$ as the closest point projection to~$K$. It has been shown in~\cite[Lemma~3.1]{DieningKreuzerSchwarzacher2013} that (see Figure~\ref{fig:PiTandS})
\begin{align}
  \label{eq:PiKfirst}
  (a - \Pi_K a) \cdot (b- \Pi_K a) \leq 0
\end{align}
for all $a\in \RRn$ and all $b \in K$. As a consequence
\begin{align*}
  \abs{\Pi_K a - \Pi_K b}^2
  &= (\Pi_K a - \Pi_K b)  \cdot (\Pi_K a - \Pi_K b)
  \\
  &= (a-b) \cdot (\Pi_K a - \Pi_K b)
  \\
  &\quad 
    + (\Pi_K a-a) \cdot (\Pi_K a - \Pi_K b) 
    + (b-\Pi_Kb) \cdot (\Pi_K a - \Pi_K b)
  \\
  &\leq
    (a-b) \cdot (\Pi_K a - \Pi_K b).
\end{align*}
This implies that for all $a,b \in \RRN$
\begin{align}
  \label{eq:PiKdifference}
  \abs{\Pi_K a - \Pi_K b}^2
  &\leq     (a-b) \cdot (\Pi_K a - \Pi_K b) \leq \abs{a-b}^2.
\end{align}
In particular, $\Pi_K$ is a contraction. For $v \in W^{1,1}(\Omega,\RRN)$ define $\Pi_K v$ by $(\Pi_K v)(x) = \Pi_K(v(x))$. Then~\eqref{eq:PiKdifference} implies
\begin{align}
  \label{eq:PiKgradient}
  \abs{\nabla (\Pi_K v)} &\leq \abs{\nabla v}.
\end{align}

\subsection{Truncation operator}
\label{sec:truncation-operator}

For $\lambda \geq 0$ we define the vectorial truncation operator $T_\lambda\,:\, \RRN \to \RRN$ by (see Figure~\ref{fig:PiTandS})
\begin{align}
  \label{eq:defTlambda}
  T_\lambda(a) 
  &\coloneqq \min \bigset{\abs{a},\lambda} \frac{a}{\abs{a}} =
    \begin{cases}
      a &\quad \text{for $\abs{a} \leq \lambda$},
      \\
      \lambda \frac{a}{\abs{a}} &\quad \text{for $\abs{a} > \lambda$}.
    \end{cases}
\end{align}
Note that $T_\lambda$ is just a special case of the projection~$\Pi_K$ defined in Section~\ref{sec:proj-conv-clos} with $K=\overline{B_\lambda(0)}$. For $v\in W^{1,1}(\Omega,\RRN)$, we define $\left(T_\lambda v\right)(x)\coloneqq T_\lambda \left(v(x)\right)$. Hence, with~\eqref{eq:PiKdifference} and \eqref{eq:PiKgradient} we obtain
\begin{align}
  \label{eq:Tlambda}
  \begin{aligned}
    \abs{T_\lambda a} &= \min \set{\abs{a},\lambda},
    \\
    \abs{T_\lambda a - T_\lambda b}^2
    &\leq     (a-b) \cdot (T_\lambda a - T_\lambda b) \leq \abs{a-b}^2,
    \\
    \abs{\nabla T_\lambda v} &\leq \abs{\nabla v}.
  \end{aligned}
\end{align}
In particular, $T_\lambda$ is a contraction. These estimates are suitable for the convex hull property.
However, for the De~Giorgi technique we need finer estimate.

\begin{lemma}
  \label{lem:Tla-Tlb-a-b}
  Let $\lambda>0$ and let $a,b\in
  \RRN$. Then
  \begin{align*}
  	(T_\lambda a - T_\lambda b)
  	\cdot (a-b) &\leq \frac 12 \left(\frac{\abs{T_\lambda a}}{\abs{a}}+\frac{\abs{T_\lambda b}}{\abs{b}}\right) \abs{a-b}^2.
  \end{align*}
\end{lemma}
\begin{proof}
  We distinguish three cases:
  \begin{enumerate}
  \item If $\abs{a},\abs{b} \leq \lambda$, then the claim
    simplifies to~$(a-b)\cdot
    (a-b) \leq \abs{a-b}^2$, which
    is obviously true.
  \item If $\abs{a}, \abs{b} \geq
    \lambda$, then the claim follows from
    \begin{align*}
      (T_\lambda a - T_\lambda
      b) \cdot (a-b) 
      &= \big(
        \tfrac{\lambda}{\abs{a}} a -
        \tfrac{\lambda}{\abs{b}}
        b \big) \cdot (a-b)
      \\
      &= \lambda \big(\tfrac{1}{\abs{a}} +
        \tfrac{1}{\abs{b}}\big) \big(
        \abs{a} \abs{b} - a \cdot b \big)
      \\
      &=  \lambda \big(\tfrac{1}{\abs{a}} +
        \tfrac{1}{\abs{b}}\big) \tfrac 12 \big( \abs{a-b}^2 - \bigabs{\abs{a} -
        \abs{b}}^2 \big)
        \\
      &\leq \tfrac 12 \big( \tfrac{\lambda}{\abs{a}} +
          \tfrac{\lambda}{\abs{b}} \big) \abs{a-b}^2.
    \end{align*}
  \item If $\abs{b} \leq \lambda < 
    \abs{a}$, then with $\gamma \coloneqq \frac{\lambda}{\abs{a}} \in [0,1)$
    \begin{align*}
      (T_\lambda a - T_\lambda
      b) \cdot (a-b) 
      &= \big( \gamma a -
        b \big) \cdot (a-b)
      \\
      &= \gamma \abs{a}^2 +
        \abs{b}^2  - (\gamma+1) a
        \cdot b
      \\
      &= \tfrac{\gamma+1}{2}
        \abs{a-b}^2 -
        \tfrac{1-\gamma}{2}
        (\abs{a}^2-\abs{b}^2).
    \end{align*}
    This, $\gamma \in [0,1)$ and $\abs{b} < \abs{a}$ proves our claim
    \begin{align*}
      (T_\lambda a - T_\lambda
      b) \cdot (a-b) 
      \leq \tfrac 12 (\gamma+1)
      \abs{a-b}^2.
    \end{align*}
  \item The case $\abs{a} \leq \lambda < \abs{b}$ follows by symmetry.
  \end{enumerate}
  This proves our claim.
\end{proof}

The chain rule implies
\begin{align}
  \label{eq:Tlambdau-gradient}
  \begin{aligned}
    \nabla T_\lambda v
    &= \indicator_{\set{\abs{v}\leq \lambda}} \nabla v + \indicator_{\set{\abs{v}>\lambda}} \frac{\lambda}{\abs{v}} \nabla v \bigg(\identity - \frac{v}{\abs{v}} \otimes \frac{v}{\abs{v}}\bigg).
  \end{aligned}
\end{align}
Hence,
\begin{align}
  \label{eq:Tlambdau-gradient-gradient}
  \nabla v : \nabla T_\lambda v
  &= \indicator_{\set{\abs{v}\leq \lambda}} \abs{\nabla v}^2 +  \indicator_{\set{\abs{v}>\lambda}} \frac{\lambda}{\abs{v}} \big( \abs{\nabla v}^2 - \abs{\nabla \abs{v}}^2\big)
\end{align}
and
\begin{align}
  \label{eq:Tlambdau-gradient2}
  \abs{\nabla T_\lambda v}^2
  &= \indicator_{\set{\abs{v}\leq \lambda}} \abs{\nabla v}^2 +  \indicator_{\set{\abs{v}>\lambda}} \bigg(\frac{\lambda}{\abs{v}}\bigg)^2 \big( \abs{\nabla v}^2 - \abs{\nabla \abs{v}}^2\big).
\end{align}
It follows from~\eqref{eq:Tlambdau-gradient-gradient} and~\eqref{eq:Tlambdau-gradient2} that
\begin{align}
  \label{eq:Tlambdau-ineq}
  \abs{\nabla T_\lambda v}^2
  &\leq
    \nabla v : \nabla T_\lambda v
    \leq \abs{\nabla v}^2.
\end{align}

\subsection{Shortening operator}
\label{sec:shortening-operator}

We define the shortening operator $S_\lambda: \RRN \to \RRN$ by
\begin{align}
  \label{eq:defSlambda}
  S_\lambda a
  &\coloneqq a - T_\lambda a
    =  \big(\abs{a}-\lambda)_+ \frac{a}{\abs{a}}
    =
     \begin{cases}
       0 &\quad \text{for $\abs{a} \leq \lambda$},
       \\
      (\abs{a}-\lambda) \frac{a}{\abs{a}} &\quad \text{for $\abs{a} > \lambda$}.
    \end{cases}
\end{align}
See Figure~\ref{fig:PiTandS} for a visualization. For $v\in W^{1,1}(\Omega,\RRN)$, we define $\left(S_\lambda v\right)(x)\coloneqq S_\lambda \left(v(x)\right)$. We obtain immediately
\begin{align}
  \label{eq:Slambdaabs}
  \abs{S_\lambda a} &= (\abs{a}-\lambda)_+,
\end{align}
By the definition of $S_\lambda$ we have for all $a,b\in \RRn$
\begin{align*}
	(a-b)(S_\lambda a-S_\lambda b) &= \abs{a-b}^2-(a-b)(T_\lambda a- T_\lambda b). 
\end{align*}
This and \eqref{eq:Tlambda} imply
\begin{align}\label{eq:Slambda}
	\abs{S_\lambda a-S_\lambda b}^2 \leq (a-b)\cdot(S_\lambda a -S_\lambda b) \leq \abs{a-b}^2.
\end{align}
In particular, $S_\lambda$ is a contraction. Again for the De Giorgi method we need some refined estimates.
\begin{lemma}
  \label{lem:Slx-Sly-x-y}
  Let $\lambda>0$ and let $a,b\in
  \RRn$. Then
  \begin{align*}
    (S_\lambda a - S_\lambda b)
    \cdot (a-b) &\geq \frac 12 \bigg(
                  \frac{\abs{S_\lambda a}}{\abs{a}} +
                  \frac{\abs{S_\lambda b}}{\abs{b}} 
                  \bigg) \abs{a-b}^2.
  \end{align*}
\end{lemma}
\begin{proof}
  Since $S_\lambda = \identity - T_\lambda$ we have
  \begin{align*}
    (S_\lambda a - S_\lambda b)
    \cdot (a-b) &= \abs{a-b}^2 -(T_\lambda a - T_\lambda b)
                  \cdot (a-b). 
  \end{align*}
  Now, the claim  follows directly from Lemma~\ref{lem:Slx-Sly-x-y}.
\end{proof}

Using $S_\lambda = \identity - T_\lambda$ we obtain from~\eqref{eq:Tlambdau-gradient}
\begin{align}
  \label{eq:Slambdau-gradient}
  \begin{aligned}
    \nabla S_\lambda v
    &= \indicator_{\set{\abs{v}>\lambda}} \bigg(  \frac{\abs{v}-\lambda}{\abs{v}}\nabla v + \frac{\lambda}{\abs{v}}\nabla v\, \bigg(\frac{v}{\abs{v}} \otimes \frac{v}{\abs{v}} \bigg)\bigg).
  \end{aligned}
\end{align}
Thus,
\begin{align}
  \label{eq:Slambdau-gradient-gradient-aux}
  \begin{aligned}
    \nabla v : \nabla S_\lambda v &= \indicator_{\set{\abs{v}> \lambda}} \bigg( \frac{\abs{v}-\lambda}{\abs{v}} \abs{\nabla v}^2 + \frac{\lambda}{\abs{v}} \bigabs{\nabla \abs{v}}^2\bigg)
  \end{aligned}
\end{align}
and
\begin{align}
  \label{eq:Slambdau-gradient2}
  \begin{aligned}
    \abs{\nabla S_\lambda v}^2 &= \indicator_{\set{\abs{v}>\lambda}} \Bigg( \bigg(\frac{\abs{v}-\lambda}{\abs{v}}\bigg)^2 \abs{\nabla v}^2 + \bigg( 1 - \bigg(\frac{\abs{v}-\lambda}{\abs{v}} \bigg)^2\bigg) \bigabs{\nabla \abs{v}}^2 \Bigg).
  \end{aligned}
\end{align}
Thus, \eqref{eq:Slambdau-gradient-gradient-aux} and \eqref{eq:Slambdau-gradient2} imply
\begin{align}
  \label{eq:Slambdau-gradient-ineq1}
    \abs{\nabla S_\lambda v}^2
  &\leq
    \nabla v : \nabla S_\lambda v
    \leq \abs{\nabla v}^2.
\end{align}
and
\begin{align}
  \label{eq:Slambdau-gradient-ineq2}
  \frac{\abs{S_\lambda v}}{\abs{v}} \abs{\nabla v} =   \frac{(\abs{v}-\lambda)_+}{\abs{v}} \abs{\nabla v} &\leq   \abs{\nabla S_\lambda v} \leq \abs{\nabla v}.
\end{align}
The most important estimate for the De Giorgi method follows by~  \eqref{eq:Slambdau-gradient-gradient-aux}
\begin{align}
  \label{eq:Slambdau-low}
  \nabla v : \nabla S_\lambda v    
  &\geq \indicator_{\set{\abs{v}> \lambda}}  \frac{\abs{v}-\lambda}{\abs{v}} \abs{\nabla v}^2 = \frac{\abs{S_\lambda v}}{\abs{v}} \abs{\nabla v}^2.
\end{align}
For all $\gamma > \lambda > 0$ and all $a \in \RRN$ with $\abs{a}\geq \gamma$ we have
\begin{align}
  \label{eq:Slambdagamma}
  \abs{a} &= \frac{\abs{a}}{\abs{a} -\lambda} \abs{S_\lambda a} \leq \frac{\gamma}{\gamma - \lambda} \abs{S_\lambda a}.
\end{align}

\section{Orlicz spaces and N-functions}
\label{sec:orlicz}
In this section we collect some background on N-functions and Orlicz spaces necessary to define the local and nonlocal systems we study. 

An \emph{N-function} (the N stands for ``nice'') is a function $\phi :[0,\infty)\rightarrow [0,\infty)$ such that $\phi$ is convex, left-continuous, $\phi(0)=0$, $\lim_{t\rightarrow 0}\frac{\phi(t)}{t}=0$ and $\lim_{t\rightarrow \infty} \frac{\phi(t)}{t}=\infty$. The most important examples are $\phi(t)=\frac 1p t^p$ for $p\in (1,\infty)$, but also more general growth conditions like for example $\phi(t)=\frac 1p t^p +\frac 1q t^q$ are included.

For an N-function $\phi$ and an open set $\Omega \subset \RRn$ the \emph{Orlicz space} is defined by
\begin{align*}
	L^\phi(\Omega,\RRN) \coloneqq \Big\{v\in L^1_\loc (\Omega,\RRN): \int_\Omega \phi\left(\frac{\abs{v(x)}}{\lambda}\right) \,dx<\infty \text{ for some } \lambda>0 \Big\},
\end{align*}
equipped with the Luxemburg norm
\begin{align*}
	\norm{v}_{L^\phi(\Omega,\RRN)}\coloneqq  \inf \Big\{ \lambda >0 : \int_\Omega \phi\left(\frac{\abs{v(x)}}{\lambda}\right) \,dx \leq 1 \Big\}.
\end{align*}
Furthermore, we define the Orlicz-Sobolev space
\begin{align*}
	W^{1,\phi}(\Omega,\RRN)\coloneqq \set{v\in L^\phi (\Omega,\RRN): \abs{\nabla v}\in L^\phi (\Omega)}
\end{align*} 
and equip it with the norm $\norm{v}_{W^{1,\phi}(\Omega,\RRN)}\coloneqq \norm{v}_{L^\phi(\Omega,\RRN)}+\norm{\nabla v}_{L^\phi(\Omega,\RRN)}$.

For every N-function $\phi$ there exists a non-decreasing right-derivative $\phi'$ satisfying $\phi(t)=\int_0^t\phi'(s)\,ds$ for all $t>0$. The \emph{conjugate N-function} $\phi^*$ is defined via
\begin{align*}
	\phi^*(s)\coloneqq \sup_{t\geq 0} (st - \phi(t)).
\end{align*}
We have $(\phi^*)^*(t)=\phi(t)$ and $(\phi^*)'(\phi'(t))=t$ for all $t\geq 0$.
The lower- and upper \emph{Simonenko indices} $p$ and $q$ of $\phi$ are defined by
\begin{align*}
	p\coloneqq \inf_{t>0}\frac{\phi'(t)t}{\phi(t)} \leq \sup_{t>0}\frac{\phi'(t)t}{\phi(t)} \eqqcolon q.
\end{align*}
For example, the N-function $\phi(t)=\frac 1p t^p + \frac 1q t^q$ with $1<p\leq q<\infty$ has indices $p$ and $q$. For every N-function with upper Simonenko index $q<\infty$ there exists a constant $c>0$ such that $\phi(2t)\leq c \phi(t)$ for all $t\geq 0$. $c$ is often called the $\Delta_2$-constant of $\phi$.

It is well known that for all $s,t \geq 0$ we have
\begin{align}\label{eq:simonenko_grwoth}
	\min \bigset{s^{p}, s^{q}} \phi(t) \leq \phi(st) \leq
	\max \bigset{s^{p}, s^{q}} \phi(t).
\end{align}
It holds
\begin{align}\label{eq:orlicz1}
	2^{-\frac{p}{p-1}}\phi(s)\leq \phi^\ast (\phi'(s))\leq 2^q \phi(s),
\end{align}
or in short $\phi^\ast (\phi'(s))\eqsim \phi(s)$. We have the following version of Young's inequality: for every $\epsilon >0$ and for all $s,t\geq 0$ we have
\begin{align}\label{eq:young_with_phi}
	\begin{aligned}
		st&\leq \epsilon\phi(s)+\epsilon^{1-\frac{p}{p-1}} \phi^\ast (t), \qquad\text{and}\\
		st&\leq \epsilon^{1-q}\phi(s)+\epsilon \phi^\ast (t) .
	\end{aligned}
\end{align}
Combining this with \eqref{eq:orlicz1} we get for all $s,t\geq 0$
\begin{align}\label{eq:young_with_phi2}
	\phi'(s)t \leq C_\epsilon\phi(s)+\epsilon\phi(t) .
\end{align}

\section{Convex hull property}
\label{sec:convex-hull-property-1}

In this section we prove the convex hull property for $\phi$-harmonic functions in the local and nonlocal case. The convex hull property is the natural extension of the maximum principle to the vectorial setting. For our proof we need uniqueness of minimizers and therefore assume that $\phi$ is strictly convex, i.e., that we have
\begin{align*}
	\phi\left(\frac{s+t}{2}\right)<\frac{\phi(s)+\phi(t)}{2} \qquad \text{for all $0\leq s<t$.}
\end{align*}

\subsection{Local case}
\label{sec:chull-local-case}

Let $\phi$ be a strictly convex N-function with Simonenko indices $1\leq p\leq q<\infty$ and $\Omega \subset \RRn$ be a bounded domain with Lipschitz boundary and let $u_D \in W^{1,\phi}(\Omega,\RRN)$. For $v \in W^{1,\phi}(\Omega,\RRN)$ we consider the energy
\begin{align*}
  \mathcal{J}_\phi(v) = \int_\Omega \phi(\abs{\nabla v})\,dx.
\end{align*}
Suppose that $u$ is $\phi$-harmonic on~$\Omega$ with boundary values~$u_D$ in the sense that
\begin{align*}
  u = \argmin_{v \in u_D + W^{1,\phi}_0(\Omega,\RRN)} \mathcal{J}_\phi(v).
\end{align*}
We define the closed convex hull of the boundary values of $u$ by
\begin{align*}
	 \cconvexhull u(\partial \Omega)\coloneqq \bigcap \set{K\subset\RRN \text{ convex, closed}: u(x)\in K\text{ for almost all $x\in \partial \Omega$}},
\end{align*}
where the ``almost'' has to be understood in the sense of $(n-1)$-dimensional Hausdorff measure.

\begin{lemma}\label{lem:convexhull_nonlocal}
  There holds
  \begin{align*}
    u(x) \in \cconvexhull u(\partial \Omega)\qquad\text{for almost all $x\in \Omega$}.
  \end{align*}
  As a consequence $\norm{u}_{L^\infty(\Omega,\RRN)} \leq
    \norm{u}_{L^\infty(\partial \Omega,\RRN)}$.
\end{lemma}
\begin{proof}
  Let $K :=\cconvexhull u(\partial \Omega)$. Then $\Pi_K u = u$ on~$\partial \Omega$, so it has the same boundary values. Due to~\eqref{eq:PiKgradient} we have $ \mathcal{J}_\phi(\Pi_K u) \leq \mathcal{J}_\phi(u)$. Since $\phi$ is strictly convex, the minimizer of $\mathcal{J}_\phi$ is unique. Hence, $u = \Pi_K u$. This proves the claim.
\end{proof}

\subsection{Nonlocal case}
\label{sec:chull-non-local-case}

Let $\Omega \subset \RRn$ be a bounded domain with Lipschitz boundary. Let $u_D \in W^{s,\phi}(\RRn,\RRN)$. For $v \in W^{s,\phi}(\RRn,\RRN)$ we consider the energy
\begin{align*}
  \mathcal{J}^s_\phi(v) = \int_{\RRn} \int_{\RRn} \phi\left(\frac{\abs{v(x)-v(y)}}{\abs{x-y}^s}\right) \dxyn 
\end{align*}
Suppose that $u$ is $\phi$-harmonic on~$\Omega$ with boundary values~$u_D$ in the sense that
\begin{align*}
  u = \argmin_{v \in u_D + W^{s,\phi}_0(\Omega,\RRN)} \mathcal{J}^s_\phi(v).
\end{align*}

We define the closed convex hull of the complement values of $u$ by
\begin{align*}
	\cconvexhull u(\Omega^c)\coloneqq \bigcap \set{K\subset\RRN \text{ convex, closed}: u(x)\in K\text{ for almost all $x\in  \Omega^c$}}.
\end{align*}

\begin{lemma}\label{lem:convexhull_local}
  There holds
  \begin{align*}
    u(x) \in\cconvexhull u(\Omega^c) \quad\text{for almost all $x\in \Omega$}.
  \end{align*}
  As a consequence $\norm{u}_{L^\infty(\Omega,\RRN)} \leq
    \norm{u}_{L^\infty(\Omega^c,\RRN)}$. 
\end{lemma}
\begin{proof}
  Let $K:=	\cconvexhull u(\Omega^c)$ denote the closed convex hull of the boundary values. Then $\Pi_K u = u$ on~$\Omega^c$, so it has the same boundary values. Moreover, due to by~\eqref{eq:PiKdifference} it has smaller or equal energy, i.e., $ \mathcal{J}^s_\phi(\Pi_K u) \leq \mathcal{J}^s_\phi(u)$. Since the minimizer is unique, we have $u = \Pi_K u$. This proves the claim.
\end{proof}

\section{De Giorgi method}
\label{sec:de-giorgi-method}

In this section we apply the vectorial De Giorgi iteration technique to show local boundedness of solutions to local and nonlocal systems. Throughout the section we assume that $\phi$ is an N-function with lower  index $p>1$ and upper Simonenko index $q<\infty$.

A key ingredient of the proofs is the following iteration lemma which  can be found, e.g., in \cite[Lemma 7.1]{Giusti03}.

\begin{lemma}[Iteration Lemma]\label{lem:iteration}
	Let $a,b\geq 1$ and $\alpha >0$. Let $(W_k)$ be a sequence of non-negative real numbers such that
	\begin{align}\label{eq:iteration_lemma}
		W_{k}\leq a \,b^{k}W_{k-1}^{1+\alpha}\quad \text{for all $k\geq 1$}.
	\end{align}
	If we additionally have that
	\begin{align*}
		W_0 < a^{-\frac 1 \alpha}b^{-\frac 1 \alpha - \frac{1}{\alpha^2}},
	\end{align*}
	then $\lim _{k\rightarrow \infty} W_k =0$.
\end{lemma}

\subsection{Local boundedness of solutions to local systems}
\label{sec:degiorgi-local-orlicz}
 
We start with the case of local systems, in particular we prove Theorem \ref{thm:orlicz_L_infty}.
As in Section~\ref{ssec:local-vect-probl} we assume that~$u$ is a local minimizer of~$\mathcal{J}_\phi$, i.e.
\begin{align}
  \label{eq:plap-scalar}
  \int A(\nabla u) : \nabla \psi\,dx &=0
\end{align}
for all $\psi \in W^{1,\phi}_0(\Omega, \RRN)$, where $A\,:\, \RRN \to \RR^{n \times N}$ is defined by
\begin{align}
  \label{eq:defA-local}
  A(Q) &\coloneqq \phi'(Q) \frac{Q}{\abs{Q}}.
\end{align}

We begin with a \Caccioppoli~type estimate for shortenings of $u$.
\begin{lemma}[\Caccioppoli~for shortenings]\label{lem:Cacc_local}
	Let $B_r$, $B_R$ be concentric balls with $r<R$ and $0<\lambda<\Lambda$. Let $u$ fulfill the assumptions of Theorem \ref{thm:orlicz_L_infty}. Then we have
	\begin{align*}
		\int_{B_r} \phi(\abs{\nabla S_\Lambda u}) \,dx\leq c \int_{B_R} \phi \left(\frac{\Lambda}{\Lambda-\lambda}\frac{\abs{S_\lambda u}}{R-r}\right) \, dx ,
	\end{align*}
	where $c>0$ depends only on the dimension $n$ and the Simonenko indices $p$ and $q$ of $\phi$.
\end{lemma}

\begin{proof}
	Let $\gamma := \frac 12 (\lambda + \Lambda)$. Let $\eta$ be a smooth cut-off function with 
	\begin{align}\label{eq:cut_off_lambda_Lambda}
		\indicator_{B_r}\leq \eta \leq \indicator_{B_R}\quad\text{and}\quad \abs{\nabla \eta}\leq \frac{c}{R-r}.
 	\end{align}
 	We use $\eta^q S_\gamma u$ as a test function to arrive at
 	\begin{align*}
 		0&= \int \phi'(\abs{\nabla u}) \frac{\nabla u}{\abs{\nabla u}}:\left(\eta^q \nabla S_\gamma u + q\eta^{q-1}\nabla \eta \otimes S_\gamma u \right)\,dx
 		\\
 		&\eqqcolon \mathrm{I}+\mathrm{II}.
 	\end{align*}
 	Using \eqref{eq:Slambdau-low} and $\phi^\prime(t)t\geq \phi(t)$ we have
 	\begin{align*}
 		\mathrm{I}&\geq \int \eta ^q \phi'(\abs{\nabla u}) \frac{\abs{S_\gamma u}}{\abs{u}}\abs{\nabla u}\,dx \geq  \int \eta ^q \phi(\abs{\nabla u}) \frac{\abs{S_\gamma u}}{\abs{u}}\,dx .
 	\end{align*}
 	By \eqref{eq:cut_off_lambda_Lambda}, Young' inequality and $\phi^\ast(\phi'(t))\eqsim \phi(t)$ we have for every $\delta >0$
 	\begin{align*}
 		\abs{\mathrm{II}} &\lesssim \int \frac{\abs{S_\gamma u}}{\abs{u}}\left( \eta^{q-1} \phi' (\abs{\nabla u}) \frac{\abs{u}}{R-r}\right)\,dx
 		\\
 		&\lesssim \delta \int \eta^q \phi(\abs{\nabla u})\frac{\abs{S_\gamma u}}{\abs{u}} \,dx + c_\delta \int_{B_R} \phi \left(\frac{\abs{u}}{R-r}\right) \frac{\abs{S_\gamma u}}{\abs{u}}\, dx .
 	\end{align*}
 	If we choose $\delta>0$ small enough, then we can absorb the first term into $\mathrm{I}$. Thus, we can summarize our findings so far as
 	\begin{align}\label{eq:local_orlicz_tested}
 		\int \eta^q \phi(\abs{\nabla u})\frac{\abs{S_\gamma u}}{\abs{u}} \,dx \lesssim \int_{B_R} \phi \left(\frac{\abs{u}}{R-r}\right) \frac{\abs{S_\gamma u}}{\abs{u}}\, dx.
 	\end{align}
 	Note that
 	\begin{align}\label{eq:orlicz_aux1}
 		\indicator_{\set{\abs{v}>\Lambda}}\abs{v} = \frac{\abs{v}}{\abs{v}-\gamma} \abs{S_\gamma v} \leq \frac{\Lambda}{\Lambda -\gamma}\abs{S_\gamma v}=  \frac{2\Lambda}{\Lambda -\lambda}\abs{S_\gamma v},
 	\end{align}
 	and
 	\begin{align}\label{eq:orlicz_aux2}
 		\indicator_{\set{\abs{v}>\gamma}}\abs{v} = \frac{\abs{v}}{\abs{v}-\lambda} \abs{S_\lambda v} \leq \frac{\gamma}{\gamma -\lambda}\abs{S_\lambda v}\leq  \frac{2\Lambda}{\Lambda -\lambda}\abs{S_\lambda v}.
 	\end{align}
 	Using \eqref{eq:orlicz_aux1}, we can estimate the left hand side of \eqref{eq:local_orlicz_tested} from below:
 	\begin{align*}
 		\int \eta^q \phi(\abs{\nabla u})\frac{\abs{S_\gamma u}}{\abs{u}} \,dx &\geq \frac{\Lambda-\lambda}{2\Lambda} \int \eta^q \phi(\abs{\nabla u}) \indicator_{\set{\abs{u}>\Lambda}}\,dx
 		\\
 		&\geq \frac{\Lambda-\lambda}{2\Lambda} \int_{B_r}  \phi(\abs{\nabla S_\Lambda u}) \,dx,
 	\end{align*}
 	where in the last step we used \eqref{eq:Slambdau-gradient-ineq2} and the fact that $\nabla S_\Lambda u$ vanishes on $\set{\abs{u}\leq \Lambda}$.
 	
 	On the other hand we can estimate the right hand side of \eqref{eq:local_orlicz_tested} with \eqref{eq:orlicz_aux2}:
 	
 	\begin{align*}
 		\int_{B_R} \phi \left(\frac{\abs{u}}{R-r}\right) \frac{\abs{S_\gamma u}}{\abs{u}}\, dx&\leq \int_{B_R} \phi' \left(\frac{\abs{u}}{R-r}\right)\frac{\abs{S_\gamma u}}{R-r} \, dx
 		\\
 		&\leq \int_{B_R} \phi' \left(\frac{1}{R-r}\frac{\Lambda}{\Lambda-\lambda}\abs{S_\lambda u}\right)\frac{\abs{S_\gamma u}}{R-r} \, dx 
 		\\
 		&\leq \int_{B_R} \phi \left(\frac{1}{R-r}\frac{\Lambda}{\Lambda-\lambda} \abs{S_\lambda u}\right)\frac{\Lambda-\lambda}{\Lambda} \,dx.
 	\end{align*}
 	Overall, we have shown
 	\begin{align*}
 		 \int_{B_r} \phi(\abs{\nabla S_\Lambda u}) \,dx\lesssim \int_{B_R} \phi \left(\frac{\Lambda}{\Lambda-\lambda}\frac{\abs{S_\lambda u}}{R-r}\right) \, dx .
 	\end{align*}
 	This finishes the proof.
 \end{proof}

We are now ready to proof our main theorem in the local case.

\begin{proof}[Proof of Theorem \ref{thm:orlicz_L_infty}]
 	We now proceed by applying Lemma \eqref{lem:Cacc_local} to certain sequences $B_k$ of decreasing concentric balls and $\lambda_k$ of increasing scalar thresholds. 
 	
 	For that let $B_k \coloneqq (1+2^{-k})B$ and $\lambda_k \coloneqq (1-2^{-k})\lambda_\infty$, where $\lambda_\infty >0$ is some fixed number we choose later. We set
 	\begin{align*}
 		U_k \coloneqq \fint_{B_k} \phi \left( \frac{\abs{S_{\lambda_k}u}}{r}\right)\,dx.
 	\end{align*}
 	This sequence is almost decreasing in the sense that for all $k\geq 1$ 
 	\begin{align*}
 		U_k \leq \frac{\abs{B_{k-1}}}{\abs{B_{k}}} U_{k-1}\leq 2^n \,U_{k-1}.
 	\end{align*}
 	Note that
 	\begin{align}\label{eq:lambdak}
 			\frac{\lambda_{k}}{\lambda_{k}-\lambda_{k-1}} &= \frac{1-2^{-k}}{(1-2^{-k})-(1-2^{-k+1})}=2^k-1\leq 2^k 
 	\end{align}
 	and
 	\begin{align}\label{eq:r_k}
 		r_{k-1}-r_k =2^{-k}r.
 	\end{align}
 	Combining this with Lemma \eqref{lem:Cacc_local}, we get for all $k\geq 1$
 	\begin{align}\label{eq:orlicz_local_upper}
 		\fint_{B_k} \phi(\abs{\nabla S_{\lambda_k} u}) \,dx\lesssim  2^{2qk} U_{k-1} .
 	\end{align}
 	On the other hand, we can use Hölder's inequality as well as Lemma \ref{lem:Sob_Ponc_Orlicz} to get
 	\begin{align*}
 		U_k &\leq \left(\fint_{B_k} \phi^{\frac{n}{n-1}}\left(\frac{\abs{S_{\lambda_k}u}}{r}\right)\,dx\right)^{\frac{n-1}{n}} \left(\frac{\abs{\set{S_{\lambda_k}u\neq 0}\cap B_k}}{\abs{B_k}}\right)^{\frac 1n}
 		\\
 		&\lesssim \bigg[\bigg(\fint_{B_k} \phi^{\frac{n}{n-1}}\left(\frac{\abs{S_{\lambda_k}u-\mean{S_{\lambda_k u}}_{B_k}}}{r}\right)\,dx\bigg)^{\frac{n-1}{n}}\!\!\!+\phi(r^{-1}\abs{\mean{S_{\lambda_k}u}_{B_k}})\bigg]
 		\\
 		&\quad\times\left(\frac{\abs{\set{S_{\lambda_k}u\neq 0}\cap B_k}}{\abs{B_k}}\right)^{\frac 1n}
 		\\
 		&\lesssim \bigg(\fint_{B_k} \phi \left(\abs{\nabla S_{\lambda_k}u}\right) \,dx+U_k\bigg) ~ \left(\frac{\abs{\set{S_{\lambda_k}u\neq 0}\cap B_k}}{\abs{B_k}}\right)^{\frac 1n}.
 	\end{align*}
 	Using $U_k \leq 2 U_{k-1}$ and  \eqref{eq:orlicz_local_upper} we have shown
 	\begin{align}\label{eq:orlicz_local_lower}
 		U_k \lesssim 2^{2qk}U _{k-1} ~ \left(r^{-n}\abs{\set{S_{\lambda_k}u\neq 0}\cap B_k}\right)^{\frac 1n}.
 	\end{align}
 	Note that
 	\begin{align*}
 		r^{-n}\bigabs{\set{S_{\lambda_k}u\neq 0}\cap B_k}&=r^{-n}\bigabs{\set{r^{-1}\abs{S_{\lambda_{k-1}}u}>r^{-1}2^{-k}\lambda_\infty}\cap B_k}
 		\\
 		&\lesssim \frac{1}{\phi(r^{-1}2^{-k}\lambda_\infty)}\fint_{B_k} \phi(\abs{r^{-1}S_{\lambda_{k-1}}u})\,dx
 		\\
 		&\lesssim \frac{1}{\phi(r^{-1}2^{-k}\lambda_\infty)} U_{k-1}
 		\\
 		&\leq \frac{2^{kq}}{\phi(r^{-1}\lambda_\infty) } U_{k-1}. 
 	\end{align*}
 	Combining this with \eqref{eq:orlicz_local_lower} we conclude
 	\begin{align}\label{eq:orlicz_local_iteration}
 		\frac{U_k}{\phi(r^{-1}\lambda_\infty)}\lesssim 2^{k(2q+\tfrac qn)} \left(\frac{U_{k-1}}{\phi(r^{-1}\lambda_\infty)}\right)^{1+\tfrac 1n}.
 	\end{align}
 	We can apply the iteration Lemma \ref{lem:iteration} to \eqref{eq:orlicz_local_iteration}. This gives us that there exists some constant $\delta>0$, depending only on $p,q$ and $n$ such that $U_0 < \delta \phi(r^{-1}\lambda_\infty)$ implies that $U_k \rightarrow 0 $ as $k\rightarrow \infty$. Thus, choosing $\lambda_\infty$ such that $U_0=2\delta \phi(r^{-1}\lambda_\infty)$, we conclude that $u \in L^{\infty}(B)$ with 
 	\begin{align*}
 		\sup_B \phi(r^{-1}\abs{u}) =\phi\big(\sup_B r^{-1} \abs{u}\big) \leq \phi(r^{-1}\lambda_\infty) = 2 \delta^{-1} U_0 = 2\delta^{-1}\fint_{2B} \phi(r^{-1}\abs{u} )\,dx .
 	\end{align*} 
 	This finishes the proof.
\end{proof}

\subsection{Local boundedness of solutions of nonlocal systems}
\label{sec:degiorgi-non-local-case}

In this section we prove local boundedness for solutions of nonlocal nonlinear equations. 

In the case of nonlocal systems, terms of the form $v(x)-v(y)$ appear frequently. Because of this it is convenient to use the following notation: Let $v$ be a (scalar- or vector-valued) function, $x,y\in \RRn$ and $s\in (0,1)$. We define the difference operator~$\delta_{x,y}$ and the symmetrization operator~$\sigma_{x,y}$ by
\begin{align*}
	\delta_{x,y} v\coloneqq v(x)-v(y), \qquad \delta_{x,y}^s v \coloneqq \frac{v(x)-v(y)}{\abs{x-y}^s}, \quad \text{and} \quad \sigma_{x,y}v \coloneqq \tfrac 12 \big(v(x)+v(y)\big).
\end{align*}
We have the product rule
\begin{align}\label{eq:product_rule}
	\delta_{x,y} (fg)= \delta_{x,y}f \sigma_{x,y} g + \sigma_{x,y} f \delta_{x,y} g .
\end{align}
Furthermore, we have
\begin{align}\label{eq:sigma_rule}
	\sigma_{x,y}(\abs{fg})\leq 2 \sigma_{x,y} \abs{f} \sigma_{x,y} \abs{g}.
\end{align}

As in Section~\ref{ssec:nonl-vect-probl} we assume that $u$ is local minimizer of~$\mathcal{J}^s_\phi$. In particular, $u$ solves
\begin{align}\label{eq:weak_equation_NL}
	\int_{\RRn} \int_{\RRn}  A(\delta_{x,y}^s u)\cdot \delta_{x,y}^s \psi  \dxyn =0,
\end{align}
for all $\psi \in W_0^{s,p}(\Omega)$, where $A\,:\, \RRN \to \RRN$ is defined by
\begin{align}
  \label{eq:defA-nonlocal}  A(a) &= \phi'(\abs{a})\frac{a}{\abs{a}}.
\end{align}
Note that in the nonlocal case $A$ maps to~$\RRN$, while in the local case~\eqref{eq:defA-local} it maps to matrices. For all $P,Q\in \RRN$ we have
\begin{align}\label{eq:A(P)-A(Q)}
	(A(P)-A(Q))\cdot(P-Q)\geq 0.
\end{align}

For a ball $B\subset \RRn$ with radius $r>0$ and center $x_B$ and $u$ we define the \emph{tail term} $\tail(u,B)$ via
\begin{align}\label{eq:tailterm}
	\tail(u,B)\coloneqq r^s \left(\phi'\right)^{-1}\left(r^s \int_{B^c}\phi'\left(\frac{\abs{u(y)}}{\abs{y-x_B}^s}\right) \frac{dy}{\abs{y-x_B}^{n+s}}\right).
\end{align}
This definition, which is also used in \cite[(3.4)]{ChakerKimWeidner2022}, is in accordance with the one used for the fractional $p$-Laplace equation (see e.g. \cite[(1.5)]{DiCastroKuusiPalatucci2016}). In \cite[(2.11)]{ByunKimOk2023} a slightly different tail term is introduced, omitting the use of~$(\phi')^{-1}$.

Note that the equation~\eqref{eq:weak_equation_NL} is invariant on the transformation $\bar{u}(x) = t^{-s} u(t x)$, i.e., if $u$ is a local minimizer on~$B$, then so is $\bar{u}$ on~$t^{-1} B$.  The expression $r^{-s} \tail(u,B)$ is also invariant under this transformation.

We now begin to work towards the proof of our main result in the nonlocal case, Theorem \ref{thm:local_boundedness_nonlocal}. The following Lemma is the nonlocal counterpart of Lemma \ref{lem:Cacc_local}.

\begin{lemma}[Nonlocal \Caccioppoli~for shortenings]\label{lem:nonlocal_Cacc}
	Let $0<\lambda <\Lambda$, let $B_r$ and $B_R$ be concentric balls in $\RRn$ with $r<R$ and let $\phi$ be an N-function with Simonenko indices $p$ and $q$, where $1<p\leq q<\infty$. Let $u\in W^{s,\phi}(\RRn,\RRN)$ be locally $\mathcal{J}^s_\phi$-harmonic in $B_R$. Then we have
	\begin{align*}
		\begin{aligned}
			\lefteqn{\int_{B_r}\int_{B_r}\phi(\abs{\delta_{x,y}^s S_\Lambda u}) \frac{dx\,dy}{\abs{x-y}^n}
			\lesssim \int_{B_R} \phi \left( \frac{\Lambda}{\Lambda-\lambda}\frac{R}{R-r}\frac{\abs{S_\lambda u}}{R^s}\right)  \,dx}\qquad\qquad&
			\\
			&+\frac{\Lambda}{\Lambda-\lambda}\left(\frac{R}{R-r}\right)^{n+sq} \frac{\phi' \left(R^{-s}\tail\left(u,B_R\right)\right)}{\phi'(\frac{\Lambda-\lambda}{\Lambda}\frac{\Lambda}{R^s})}\int_{B_R} \phi\left(\frac{\abs{S_\lambda u}}{R^s}\right)\,dx.
		\end{aligned}
	\end{align*}
	where $c>0$ depends only on $p$, $q$, $s$ and $n$.
\end{lemma}

\begin{proof}
	Define $\gamma\coloneqq \frac 12 (\lambda + \Lambda)$ and let $\eta$ be a smooth cut-off function with
	\begin{align*}
		\indicator _{B_{\frac{2r+R}{3}}}\leq \eta \leq \indicator_{B_{\frac{r+2R}{3}}} \qquad \text{and} \qquad \abs{\nabla \eta} \leq \frac{c}{R-r}.
	\end{align*}
	Recall that for all $\psi \in W^{s,p}_0 (B_R)$
	\begin{align*}
		\int \int A(\delta_{x,y}^s u)\cdot \delta_{x,y} ^s \psi \dxyn =0 .
	\end{align*}
	Using $\psi =0$ on $(B_R)^c$ and symmetry in the variables $x$ and $y$ yields
	\begin{align*}
		\int_{B_R}\int_{B_R} A(\delta_{x,y}^s u)\cdot\delta_{x,y}^s \psi \dxyn = -2 \int_{(B_R)^c}\int_{B_R}  A(\delta_{x,y}^s u)\cdot \delta_{x,y}^s \psi \dxyn .
	\end{align*}
	We apply this with $\psi =\eta^q S_\gamma u$ and use the product rule \eqref{eq:product_rule} to get
	\begin{align*}
			\mathrm{I}&\coloneqq \int_{B_R}\int_{B_R} \sigma_{x,y} (\eta^q)A(\delta_{x,y}^s u)\cdot \delta_{x,y}^s S_\gamma u \dxyn
			\\
			&= -\! \int_{B_R}\int_{B_R}\delta_{x,y}^s (\eta^q) A(\delta_{x,y}^s u)\cdot  \sigma_{x,y}S_\gamma u \dxyn - 2\!\!\int_{(B_R)^c}\int_{B_R}  A(\delta_{x,y}^s u)\cdot \delta_{x,y}^s \psi \dxyn
			\\
			&\coloneqq \mathrm{II}+\mathrm{III} .
	\end{align*}
	From Lemma \ref{lem:Slx-Sly-x-y} we know
	\begin{align}\label{eq:Slambda_good_notation}
		\delta_{x,y}S_\gamma u \cdot \delta_{x,y}u \geq  \sigma_{x,y} \bigg(\frac{\abs{S_\gamma u}}{\abs{u}}\bigg)\abs{\delta_{x,y} u}^2.
	\end{align}
	Thus, we can estimate $\mathrm{I}$ from below as follows.
	\begin{align*}
		I&=\int_{B_R}\int_{B_R} \frac{\phi'(\abs{\delta_{x,y}^su})}{\abs{\delta_{x,y}^s u}}\sigma_{x,y} (\eta^q) \frac{\delta_{x,y} S_\gamma u\cdot \delta_{x,y} u}{\abs{x-y}^{2s}} \dxyn
		\\
		&\geq \int_{B_R}\int_{B_R} \phi'(\abs{\delta_{x,y}^su})\abs{\delta_{x,y}^s u}\sigma_{x,y} (\eta^q) \sigma_{x,y} \bigg(\frac{\abs{S_\gamma u}}{\abs{u}}\bigg) \dxyn
		\\
		&\geq \int_{B_R}\int_{B_R} \phi(\abs{\delta_{x,y}^su})(\sigma_{x,y} \eta)^q \sigma_{x,y} \bigg(\frac{\abs{S_\gamma u}}{\abs{u}}\bigg) \dxyn .
	\end{align*}
	We can now use \eqref{eq:orlicz_aux1} like in the local case. This gives us
	\begin{align}
		\begin{aligned}
			\mathrm{I}&\geq \frac{\Lambda-\lambda}{2\Lambda} \int_{B_R}\int_{B_R} \phi(\abs{\delta_{x,y}^su})(\sigma_{x,y} \eta)^q \sigma_{x,y}\big(\indicator_{\set{\abs{u}>\Lambda}}\big) \dxyn.
		\end{aligned}
	\end{align}
	Using that $S_\Lambda$ is a contraction and that $\sigma_{x,y}\left(\indicator_{\set{\abs{u}>\Lambda}}\right)=0$ only holds when we also have $\delta_{x,y}^s S_\Lambda u=0$, we get
	\begin{align}\label{eq:nonlocal_aux1}
		\mathrm{I}\gtrsim \frac{\Lambda-\lambda}{\Lambda} \int_{B_R}\int_{B_R} \phi(\abs{\delta_{x,y}^s S_\Lambda u})(\sigma_{x,y} \eta)^q \dxyn.
	\end{align}
	
	Next, we estimate $\mathrm{II}$. We notice that
	\begin{align*}
		\abs{\delta^s_{x,y} \eta^q} \lesssim \abs{\delta^s_{x,y}\eta}(\sigma_{x,y}\eta)^{q-1}.
	\end{align*}
	This, together with \eqref{eq:sigma_rule} gives
	\begin{align*}
		\abs{\mathrm{II}}&\lesssim \int_{B_R}\int_{B_R} \phi'(\abs{\delta_{x,y}^s u}) (\sigma_{x,y}\eta)^{q-1}\abs{\delta_{x,y}^s \eta} \abs{\sigma_{x,y} S_\gamma u}\dxyn
		\\
		&\leq 2 \int_{B_R}\int_{B_R} \sigma_{x,y}\bigg(\frac{\abs{S_\gamma u}}{\abs{u}}\bigg) \phi'(\abs{\delta_{x,y}^s u})(\sigma_{x,y}\eta)^{q-1}\sigma_{x,y}\abs{u}\,\abs{\delta_{x,y}^s \eta} \dxyn.
	\end{align*}
	We now use Young's inequality \eqref{eq:young_with_phi2}, \eqref{eq:simonenko_grwoth} and \eqref{eq:orlicz1}. This gives for every $\epsilon >0$ that
	\begin{align*}
		\mathrm{II}&\lesssim \int_{B_R}\int_{B_R} \sigma_{x,y}\bigg(\frac{\abs{S_\gamma u}}{\abs{u}}\bigg) \left[\epsilon\, (\sigma_{x,y}\eta)^q\phi(\abs{\delta_{x,y}^su})+C_\epsilon\phi\left(\sigma_{x,y}\abs{u}\,\abs{\delta_{x,y}^s \eta}\right)\right] \dxyn 
		\\
		&\lesssim \epsilon\,\mathrm{I} + C_\epsilon \int_{B_R}\int_{B_R} \sigma_{x,y}\left(\frac{\abs{S_\gamma u}}{\abs{u}}\right) \phi\left(\sigma_{x,y}\abs{u}\,\abs{\delta_{x,y}^s \eta}\right)\dxyn .
	\end{align*}
	Choosing $\epsilon$ small enough we have
	\begin{align}\label{eq:nonlocal_aux2}
		\begin{aligned}
			\abs{\mathrm{II}}&\leq \tfrac 12 \, \mathrm{I}+ C \int_{B_R}\int_{B_R} \sigma_{x,y}\left(\frac{\abs{S_\gamma u}}{\abs{u}}\right) \phi\left(\sigma_{x,y}\abs{u}\,\abs{\delta_{x,y}^s \eta}\right)\dxyn 
			\\
			&\coloneqq \tfrac 12 \,\mathrm{I} + C\, \widetilde{\mathrm{II}}.
		\end{aligned}
	\end{align}
	Our next step is to desymmetrize $\widetilde{\mathrm{II}}$. For this we note that
	\begin{align*}
		\frac{\abs{S_\gamma u(x)}}{\abs{u(x)}}\leq \frac{\abs{S_\gamma u(y)}}{\abs{u(y)}} \quad \text{if and only if}\quad \phi(\abs{u(x)}\abs{\delta_{x,y}^s\eta})\leq \phi(\abs{u(y)}\abs{\delta_{x,y}^s\eta}),
	\end{align*}
	and thus
	\begin{align}\label{eq:nonlocal_aux3}
		\widetilde{\mathrm{II}} &\lesssim \int_{B_R}\int_{B_R} \frac{\abs{S_\gamma u(x)}}{\abs{u(x)}} \phi\left(\abs{u(x)}\abs{\delta_{x,y}^s \eta}\right)\dxyn .
	\end{align}
	Using the fundamental theorem of calculus we know
	\begin{align*}
		\abs{\delta_{x,y}^s\eta} \leq \abs{x-y}^{1-s}\norm{\nabla \eta}_{L^\infty} \lesssim \frac{\abs{x-y}^{1-s}}{R-r}\eqsim \left(\frac{\abs{x-y}}{2R}\right)^{1-s}\frac{R^{1-s}}{R-r}.
	\end{align*}
	Inserting this into \eqref{eq:nonlocal_aux3} and using convexity of $\phi$, we get
	\begin{align*}
		\widetilde{\mathrm{II}}&\lesssim \int_{B_R}\int_{B_R} \frac{\abs{S_\gamma u(x)}}{\abs{u(x)}} \left(\frac{\abs{x-y}}{2R}\right)^{1-s} \phi\left( \frac{R^{1-s}}{R-r}\abs{u(x)}\right)\dxyn
		\\
		&\lesssim R^{-(1-s)} \int_{B_R}\frac{\abs{S_\gamma u(x)}}{\abs{u(x)}}\, \phi\left( \frac{R^{1-s}}{R-r}\abs{u(x)}\right)  \int_{B_R} \frac{dy}{\abs{x-y}^{n+s-1}}\,dx
		\\
		&\lesssim \int_{B_R}\frac{\abs{S_\gamma u(x)}}{\abs{u(x)}} \,\phi\left(\frac{R^{1-s}}{R-r}\abs{u(x)}\right)\,dx.
	\end{align*}
	As in the local case (Theorem \ref{thm:orlicz_L_infty}) we now proceed by applying \eqref{eq:orlicz_aux2}. This gives
	\begin{align*}
		\widetilde{\mathrm{II}}&\lesssim \frac{R^{1-s}}{R-r}\int_{B_R} \abs{S_\gamma u}  \,\phi'\left(\frac{R^{1-s}}{R-r}\abs{u}\right) \,dx
		\\
		&=\frac{R^{1-s}}{R-r}\int_{B_R} \abs{S_\gamma u}  \,\phi'\left(\frac{R^{1-s}}{R-r}\abs{u}\indicator_{\set{\abs{u}>\gamma}}\right) \,dx
		\\
		&\leq \frac{R^{1-s}}{R-r}\int_{B_R} \abs{S_\gamma u}  \,\phi'\left(\frac{R^{1-s}}{R-r}\frac{2\Lambda}{\Lambda -\lambda}\abs{S_\lambda u}\right) \,dx
		\\
		&\lesssim \frac{\Lambda-\lambda}{\Lambda}\int_{B_R}  \phi\left(\frac{R^{1-s}}{R-r}\frac{\Lambda}{\Lambda -\lambda}\abs{S_\lambda u}\right) \,dx .
	\end{align*}
	Combining this with \eqref{eq:nonlocal_aux3}, we arrive at
	\begin{align}\label{eq:nonlocal_aux4}
		\abs{\mathrm{II}}\leq \tfrac 12 \mathrm{I} + C\, \frac{\Lambda-\lambda}{\Lambda}\int_{B_R}  \phi\left(\frac{R^{1-s}}{R-r}\frac{\Lambda}{\Lambda -\lambda}\abs{S_\lambda u}\right) \,dx .
	\end{align}

	Finally, we turn to $\mathrm{III}$. Since $\psi =0$ on $(B_{\frac 13 r + \frac 23R})^c$, we have
	\begin{align*}
		\mathrm{III}&= -2 \int_{(B_R)^c} \int_{B_{\frac{r+2R}{3}}} A(\delta_{x,y}^s u)\cdot\frac{\psi(x)}{\abs{x-y}^s}\dxyn
		\\
		&=-2 \int_{(B_R)^c} \int_{B_{\frac{r+2R}{3}}} A(\delta_{x,y}^s u)\cdot\frac{u(x)}{\abs{x-y}^s}\eta^q(x)\frac{\abs{S_\gamma u(x)}}{\abs{u(x)}}\dxyn.
	\end{align*}
	Note that by \eqref{eq:A(P)-A(Q)}
	\begin{align*}
		A(\delta_{x,y}^s u)\cdot \frac{u(x)}{\abs{x-y}^s}&= \left(A(\delta_{x,y}^s u)-A\left(\frac{-u(y)}{\abs{x-y}^s}\right)\right)\cdot\left(\delta_{x,y}^s u
		-\frac{-u(y)}{\abs{x-y}^s}\right)
		\\
		&\quad+A\left(\frac{-u(y)}{\abs{x-y}^s}\right)\cdot\frac{u(x)}{\abs{x-y}^s}
		\\
		&\geq A\left(\frac{-u(y)}{\abs{x-y}^s}\right)\cdot\frac{u(x)}{\abs{x-y}^s}.
	\end{align*}
	Thus,
	\begin{align*}
		\mathrm{III} &\leq 2\int_{(B_R)^c} \int_{B_{\frac{r+2R}{3}}} A\left(-\frac{u(y)}{\abs{x-y}^s}\right)\cdot \frac{-u(x)}{\abs{x-y}^s} \eta^q (x) \frac{\abs{S_\gamma u(x)}}{\abs{u(x)}}\dxyn
		\\
		&\leq 2\int_{B_{\frac{r+2R}{3}}}\int_{(B_R)^c} \phi'\left(\frac{\abs{u(y)}}{\abs{x-y}^s}\right)\frac{dy}{\abs{x-y}^{n+s}}\abs{S_\gamma u(x)}\,dx .
	\end{align*}
	Since $x\in B_{\frac{r+2R}{3}}$ and $y\notin B_R$, we have
	\begin{align*}
		\frac{R-r}{3R}\abs{y-x_B}\leq \abs{y-x} \leq 2 \abs{y-x_B},
	\end{align*}
	where $x_B$ denotes the center of $B_r$ and $B_R$. We thus get
	\begin{align}\label{eq:nonlocal_aux5}
		\begin{aligned}
			\mathrm{III}&\lesssim \left(\frac{3R}{R-r}\right)^{n+s}\!\!\int_{(B_R)^c}\!\! \phi'\left( \left(\frac{3R}{R-r}\right)^{s}\frac{\abs{u(y)}}{\abs{y-x_B}^s}\right)\frac{dy}{\abs{y-x_B}^{n+s}}\int_{B_R}\abs{S_\gamma u} \,dx
		\\
		&\lesssim \left(\frac{R}{R-r}\right)^{n+sq} \phi' \left(R^{-s}\tail\left(u,B_R\right)\right) \int_{B_R} \frac{\abs{S_\lambda u}}{R^s}\indicator_{\set{\abs{u}>\gamma}}\,dx.
		\end{aligned}
	\end{align}
	On the set $\set{\abs{u}>\gamma}$ we have $\abs{S_\lambda u} > \gamma - \lambda =\frac{\Lambda-\lambda }{2}$. This gives us
	\begin{align*}
		\mathrm{III}&\leq\left(\frac{R}{R-r}\right)^{n+sq} \frac{\phi' \left(R^{-s}\tail\left(u,B_R\right)\right)}{\phi'(R^{-s}\frac 12(\Lambda-\lambda))}\int_{B_R} \frac{\abs{S_\lambda u}}{R^s}\phi'\left(\frac{\abs{S_\lambda u}}{R^s}\right)\,dx
		\\
		&\lesssim \left(\frac{R}{R-r}\right)^{n+sq} \frac{\phi' \left(R^{-s}\tail\left(u,B_R\right)\right)}{\phi'(R^{-s}(\Lambda -\lambda))}\int_{B_R} \phi\left(\frac{\abs{S_\lambda u}}{R^s}\right)\,dx.
	\end{align*}
	Combining this with our estimates from \eqref{eq:nonlocal_aux1} and \eqref{eq:nonlocal_aux4}, we have shown
	\begin{align*}
		\begin{aligned}
			\lefteqn{\frac{\Lambda-\lambda}{\Lambda}\int_{B_r}\int_{B_r}\phi(\abs{\delta_{x,y}^s S_\Lambda u}) \frac{dx\,dy}{\abs{x-y}^n}
			\lesssim \frac{\Lambda-\lambda}{\Lambda} \int_{B_R} \phi \left( \frac{\Lambda}{\Lambda-\lambda}\frac{R}{R-r}\frac{\abs{S_\lambda u}}{R^s}\right)  \,dx}\qquad\qquad&
			\\
			&+\left(\frac{R}{R-r}\right)^{n+sq} \frac{\phi' \left(R^{-s}\tail\left(u,B_R\right)\right)}{\phi'(R^{-s}(\Lambda -\lambda))}\int_{B_R} \phi\left(\frac{\abs{S_\lambda u}}{R^s}\right)\,dx.
		\end{aligned}
	\end{align*}
	This finishes the proof.
\end{proof}

Next we combine Lemma \ref{lem:nonlocal_Cacc} with the nonlocal improved \Poincare estimate Lemma \ref{lem:Sob_Ponc_OrliczNL} and use an iteration argument to show local boundedness.

\begin{proof}[Proof of Theorem \ref{thm:local_boundedness_nonlocal}]
	As in the scalar case, we apply the \Caccioppoli~type estimate Lemma \ref{lem:nonlocal_Cacc} to sequences of decreasing balls and increasing thresholds. For that let $B_k \coloneqq (1+2^{-k})B$ and $\lambda_k \coloneqq (1-2^{-k})\lambda_\infty$, where $\lambda_\infty>0$ will be chosen later. The quantity we want to iterate is
	\begin{align*}
		U_k \coloneqq \fint_{B_k} \phi\left(\frac{\abs{S_{\lambda_k u}}}{r^s}\right) \,dx.
	\end{align*}
	The sequence $(U_k)$ is almost decreasing in the sense that $U_k \leq 2^n U_{k-1}$. Recalling \eqref{eq:lambdak} and \eqref{eq:r_k}, we get from Lemma \ref{lem:nonlocal_Cacc} that
	\begin{align}\label{eq:nonlocal_Uk_upper}
		\begin{aligned}
			\fint_{B_k}\int_{B_k} \phi \left(\abs{\delta_{x,y}^s S_{\lambda_k}u}\right)\dxyn &\lesssim \left(2^{2qk}+ 2^{k(n+sq+1)} \frac{\phi'(r^{-s}\tail (u,B))}{\phi'(r^{-s}2^{-k}\lambda_\infty)}\right) U_{k-1}  
			\\
			&\leq 2^{k(2q+n+1)}\left(\frac{\phi'(r^{-s}\tail (u,B))}{\phi'(r^{-s}\lambda_\infty)}+1\right)U_{k-1}.
		\end{aligned}
	\end{align}
  We have by Hölder's inequality
	\begin{align}\label{eq:nonlocal_Uk_lower_aux}
		U_k&\leq \left(\fint_{B_k}\phi^{\frac{n}{n-\frac s2}}\left(\frac{\abs{S_{\lambda_k}u}}{r^s}\right)\,dx\right)^{\frac{n-\frac s2}{n}}\left(\frac{\abs{\set{S_{\lambda_k}u\neq 0}\cap B_k}}{\abs{B_k}}\right)^{\frac{s}{2n}}.
	\end{align} 
	Moreover, by Lemma \ref{lem:Sob_Ponc_OrliczNL} with $\alpha= \frac s2$ we have
	\begin{align*}
		\Bigg(\fint_{B_k}\phi^{\frac{n}{n-\frac s2}}\bigg(\frac{\abs{S_{\lambda_k}u}}{r^s}\bigg)\,dx\Bigg)^{\frac{n-\frac s2}{n}}&\leq \Bigg[\fint_{B_k}\phi^{\frac{n}{n-\frac s2}}\bigg(\frac{\abs{S_{\lambda_k}u-\mean{S_{\lambda_k}u}_{B_k}}}{r^s}\bigg)\,dx\Bigg]^{\frac{n-\frac s2}{n}} 
		\\
		&\quad+\phi\left(\frac{\abs{\mean{S_{\lambda_k}u}_{B_k}}}{r^s}\right)
		\\
		&\lesssim  \fint_{B_k}\int_{B_k} \phi \left(\abs{\delta_{x,y}^s S_{\lambda_k}u}\right)\dxyn +U_k.
	\end{align*}
	Combining this with \eqref{eq:nonlocal_Uk_upper}, \eqref{eq:nonlocal_Uk_lower_aux} and using $U_k\lesssim U_{k-1}$, we arrive at
	\begin{align}\label{eq:nonlocal_Uk}
		U_k \lesssim 2^{k(2q+n+1)}\left(\frac{\phi'(r^{-s}\tail (u,B))}{\phi'(r^{-s}\lambda_\infty)}+1\right)U_{k-1} \left(\frac{\abs{\set{S_{\lambda_k}u\neq 0}\cap B_k}}{\abs{B_k}}\right)^{\frac{s}{2n}}.
	\end{align}
	If $S_{\lambda_k}u \neq 0$, then $\abs{S_{\lambda_{k-1}}u}\geq 2^{-k}\lambda_\infty$. Thus,
	\begin{align*}
		\frac{\abs{\set{S_{\lambda_k}u\neq 0}\cap B_k}}{\abs{B_k}} 
		&\lesssim \frac{1}{\phi(r^{-s}2^{-k}\lambda_\infty)}\fint_{B_k}\phi (r^{-s}\abs{S_{\lambda_{k-1}}u})\,dx
		\leq \frac{2^{kq}}{\phi(r^{-s}\lambda_\infty)}U_{k-1}.
	\end{align*}
	Thus, \eqref{eq:nonlocal_Uk} implies
	\begin{align*}
		\frac{U_k}{\phi(r^{-s}\lambda_\infty)}\lesssim 2^{k(2q+n+1+\frac{s}{2n})}\left(\frac{\phi'(r^{-s}\tail (u,B))}{\phi'(r^{-s}\lambda_\infty)}+1\right)\left(\frac{U_{k-1}}{\phi(r^{-s}\lambda_\infty)}\right)^{1+\frac{s}{2n}}.
	\end{align*}
	If we choose $\lambda_\infty$ big enough such that $\lambda_\infty\geq \tail (u,B)$, then this simplifies to
	\begin{align*}
		\frac{U_k}{\phi(r^{-s}\lambda_\infty)}\lesssim 2^{k(2q+n +\frac{s}{2n})}\left(\frac{U_{k-1}}{\phi(r^{-s}\lambda_\infty)}\right)^{1+\frac{s}{2n}}.
	\end{align*}
  We now apply Lemma~\ref{lem:iteration} with
  \begin{align*}
    W_k &\coloneqq 		\frac{U_k}{\phi(r^{-s}\lambda_\infty)}, &
    \alpha &= \frac{s}{2n}, &
    b &= 2^{2q+n+1+\frac{s}{2n}}.
  \end{align*}
  This gives us that $W_k \to 0$ and therefore $U_k \to 0$ provided that $W_0 \leq \epsilon= \epsilon(n,s,q)$. This holds, e.g., if we choose $\lambda_\infty$ implicitly via
	\begin{align*}
		\phi \left(r^{-s}\lambda_\infty\right) \coloneqq \epsilon^{-1} U_0+ \phi \left(r^{-s}\tail (u,B)\right).
	\end{align*}
  It follows from $U_k \to 0$ that $S_{\lambda_\infty}u=0$ on~$B$ and therefore $\abs{u} \leq \lambda_\infty$ on~$B$. Thus,
	\begin{align}
    \begin{aligned}
      \sup_B \phi (r^{-s}\abs{u}) &\leq \phi (r^{-s}\lambda_\infty)
      \\
      &\lesssim \epsilon^{-1} U_0 + \phi (r^{-s}\tail(u,B))
      \\
      &\leq \fint_{2B} \phi (r^{-s}\abs{u})\,dx + \phi \big(r^{-s}\tail(u,B)\big).
    \end{aligned}
	\end{align}
	This finishes the proof.
\end{proof}
If is clear from the proof of Lemma~\ref{lem:nonlocal_Cacc} and Theorem \ref{thm:local_boundedness_nonlocal} that we can reduce the assumption $u\in W^{s,\phi}(\RRn,\RRN)$ to  the weaker assumption $u\in W^{s,\phi}(2B ,\RRN)$ and $\tail (u,B) < \infty$, which proves Remark~\ref{rem:moritz-tail-spaces}.

\appendix
\section{Improved \Poincare{} estimate}\label{sec:appendix}

%
%
%

We collect some auxiliary lemmas that are needed throughout the proofs.

\begin{lemma}\label{lem:Sob_Ponc_Orlicz}
	Let $\phi$ be an N-function, $v\in W^{1,\phi}(B)$, where $B\subset\RRn$ is a ball. Then
	\begin{align*}
		\left(\fint_B  \phi ^{\frac{n}{n-1}}(\abs{v-\mean{v}_B})\,dx\right)^{\frac{n-1}{n}}\leq c \fint_{B} \phi(r\abs{\nabla v}) \,dx ,
	\end{align*}
	where $c>0$ depends only on $n$ and the $\Delta_2$-constant of $\phi$.
\end{lemma}

\begin{proof}
	Set $w\coloneqq \phi(\abs{v-\mean{v}_B})$. Then
	\begin{align*}
		\abs{\nabla w} = \phi' (\abs{v-\mean{v}_B})\abs{\nabla \abs{v}}\leq  \phi' (\abs{v-\mean{v}_B})\abs{\nabla v} .
	\end{align*}
	Combining this with Sobolev's inequality and Young's inequality gives
	\begin{align*}
		\left(\fint_B  \abs{w} ^{\frac{n}{n-1}}\,dx\right)^{\frac{n-1}{n}}&\lesssim \left(\fint_B  (\abs{w-\mean{w}_B}) ^{\frac{n}{n-1}}\,dx\right)^{\frac{n-1}{n}} + \abs{\mean{w}_B}
		\\
		&\lesssim \fint_{B}  r\abs{\nabla w} + \abs{w}\,dx 
		\\
		&\lesssim \fint_B \phi(r\abs{\nabla v}) \,dx + \fint_B \phi(\abs{v-\mean{v}_B})\,dx. 
	\end{align*}
	We use \Poincare 's inequality to bound the last integral. This finishes the proof.
\end{proof}


The following nonlocal improved \Poincare{} is similar to~\cite[Lemma~4.1]{ByunKimOk2023} but with an improved range of~$\alpha$ and a simplified proof. Note that optimal global Sobolev embeddings for~$\mathcal{J}^s_\phi$ can be found in~\cite[Theorem~6.1]{AlbericoCianciPickSlavikova21}.
\begin{lemma}\label{lem:Sob_Ponc_OrliczNL}
	Let $\phi$ be an N-function, $B\subset \RRn$ be a ball and $v\in W^{s,\phi}(B)$. Then for every $\alpha\in [0,s)$ we have
	\begin{align*}
		\left(\fint_B \left(\phi(\abs{v-\mean{v}_B})\right)^{\frac{n}{n-\alpha}}\,dx \right)^{\frac{n-\alpha}{n}} \leq c \fint_{B}\int_B \phi(r^s\abs{\delta_{x,y}^sv}) \dxyn ,
	\end{align*}
	where $c>0$ depends on $n$, the $\Delta_2$-constant of $\phi$, $s$ and $\alpha$.
\end{lemma}

\begin{proof}
	Set $w\coloneqq \phi (\abs{v-\mean{v}_B})$. Then
	\begin{align}\label{eq:appendix1}
		\delta_{x,y} w \leq \big(\sigma_{x,y} \phi'(\abs{v-\mean{v}_B})\big) \abs{\delta_{x,y} v}.
	\end{align}
	Thus, by the fractional Sobolev embedding $W^{\alpha,1}(B)\embedding  L^{\frac{n}{n-\alpha}}$ we have
	\begin{align}\label{eq:appendix2}
		\begin{aligned}
			\left(\fint_B w^{\frac{n}{n-\alpha}}\,dx \right)^{\frac{n-\alpha}{n}}&\lesssim \left(\fint_B \abs{w-\mean{w}_B}^{\frac{n}{n-\alpha}}\,dx \right)^{\frac{n-\alpha}{n}}+ \mean{w}_B
		\\
		&\lesssim r^{\alpha}\fint_B\int_B \abs{\delta_{x,y}^{\alpha} w} \dxyn  + \mean{w}_B .
		\end{aligned}
	\end{align}
	Using the fractional \Poincare inequality we have
	\begin{align*}
		\mean{w}_B \leq \fint_B \phi(\abs{v-\mean{v}_B})\,dx \lesssim \fint_B\int_B \phi(r^s \abs{\delta_{x,y}^s v}) \dxyn .
	\end{align*}
	The estimate \eqref{eq:appendix1} together with Young's inequality \eqref{eq:young_with_phi2} gives us
	\begin{align*}
		r^{\alpha}\fint_B\int_B \abs{\delta_{x,y}^{\alpha} w} \dxyn &\leq r^{\alpha}\fint_B\int_B \phi'(\abs{v(x)-\mean{v}_B})\abs{\delta_{x,y}^{\alpha} v} \dxyn
		\\
		&\lesssim \fint_B \int_B \phi(\abs{\delta_{x,y}^{\alpha}v}\abs{x-y}^{\alpha-s}r^s) \frac{\abs{x-y}^{s-\alpha}}{r^{s-\alpha}}\dxyn \\
		&\quad+ \fint_B \int_B \phi(\abs{v(x)-\mean{v}_B}) \frac{\abs{x-y}^{s-\alpha}}{r^{s-\alpha}}\dxyn
		\\
		&\lesssim \fint_B\int_B \phi(r^s\abs{\delta_{x,y}^{s}v}) \dxyn +\fint_B \phi(\abs{v-\mean{v}_B}) \,dx .
	\end{align*} 
	where we used $\abs{x-y}\leq 2r$ in the last step. Combining this with the fractional \Poincare inequality and \eqref{eq:appendix2}, yields the claimed estimate.
\end{proof}

\printbibliography

\end{document}